\documentclass[12pt]{amsart}

\usepackage{amssymb}

\usepackage{amsmath, amsthm, amssymb, amsfonts}
\usepackage{amsxtra, amscd, geometry, graphicx}

\theoremstyle{plain}

\theoremstyle{definition}

\newtheorem{thm}{Theorem}
\newtheorem{lemma}[thm]{Lemma}

\newtheorem{cor}[thm]{Corollary}
\newtheorem{prop}[thm]{Proposition}
\newtheorem{conj}[thm]{Conjecture}
\newtheorem{defn}[thm]{Definition}

\newcommand{\stem}{\textrm{stem}}

\newcommand{\ran}{\textrm{ran}}
\providecommand{\Dnd}{{\mathbb{D}_\textrm{nd}}}
\providecommand{\D}{\mathbb{D}}
\providecommand{\Dtree}{{\mathbb{D}_\textrm{tree}}}
\providecommand{\stem}{{\textrm{stem}}}
\providecommand{\dnd}{{d}_\textrm{nd}}
\begin{document}

\author[Palumbo]{Justin Palumbo}
\address{Justin Palumbo, Department of Mathematics,
                    University of California at Los Angeles,
                    Los Angeles, California}
\email{justinpa@math.ucla.edu}

\begin{abstract}
We give results exploring the relationship between dominating and unbounded reals in Hechler extensions, as well as
the relationships among the extensions themselves. We show that in the standard Hechler extension there is
an unbounded real which is dominated by every dominating real, but that this fails to hold in the tree Hechler extension.
We prove a representation theorem for dominating reals in the standard Hechler extension: every dominating real eventually
dominates a sandwich composition of the Hechler real with two ground model reals that monotonically converge to infinity. We apply our
results to negatively settle a conjecture of Brendle and L{\"o}we (Conjecture 15 of \cite{BrL}). We also answer a question due to Laflamme.
\end{abstract}

\title{Unbounded and Dominating Reals in Hechler Extensions}

\maketitle

\section{Introduction}

Forcing to add dominating reals is by now a ubiquitous technique in the study of the set theory of the reals and Hechler forcing is the most basic method for adding a dominating real to the universe. Three variations of Hechler forcing have been considered in the literature. Notationally little distinction has been made between them; all three have been commonly referred to as Hechler forcing and designated by the symbol $\mathbb{D}$. In this paper we refer to them in words as the original Hechler forcing, the non-decreasing Hechler forcing and the tree Hechler forcing and symbolically we use $\D$, $\Dnd$ and $\Dtree$, respectively.

Brendle, Judah and Shelah \cite{BrJS} used a rank analysis of $\Dnd$ originally due to Baumgartner and Dordal \cite{BD} to analyze the combinatorial consequences of forcing with $\Dnd$. They showed that in $V^\Dnd$ there is a MAD family of size $\omega_1$ and a Luzin set of size $2^\omega$. The existence of the latter implies that $\textrm{non}(\mathcal{M})=\omega_1$ and $\textrm{cov}(\mathcal{M})=2^\omega$ and thus completely determines Cicho\'{n}'s diagram of cardinal characteristics. They also showed how one can modify the rank analysis of $\Dnd$ to analyze $\D$ and prove that such objects exist in $V^\D$ as well. There is also a rank analysis for $\Dtree$ (see the definitions just before Theorem 12 in \cite{BrL}). The rank analysis for $\Dtree$ is simpler than for either $\D$ or $\Dtree$ and it is not hard to see that the same Brendle, Judah and Shelah arguments go through for $V^\Dtree$ as well.
 
Since all three forcings have the same effect on the standard cardinal characteristics of the continuum it is only natural to ask if all three forcings are the same. In this paper we will show that while $\D$ and $\Dnd$ are equivalent from the forcing point of view (and thus we may safely use the term Hechler forcing for both), $\Dtree$ is different. To accomplish this we compare the unbounded and dominating reals in $V^\D$ and $V^\Dtree$, and prove the following two theorems.

\begin{thm} \label{DtreeNotDom}
Let $d$ a be $\Dtree$-generic real over $V$. Then for any unbounded real $x\in\omega^\omega\cap V[d]$ there is some dominating real $y\in\omega^\omega\cap V[d]$ so that $x$ is not eventually dominated by $y$.
\end{thm}

\begin{thm} \label{DDom}
Let $d$ be a $\D$-generic real over $V$. Then there is an unbounded real $x\in\omega^\omega\cap V[d]$ so that for every dominating real $y\in\omega^\omega\cap V[d]$ we have that $x$ is eventually dominated by $y$.
\end{thm}

Thus the two forcings are not equivalent. We will derive Theorem \ref{DDom} from the following theorem, which we consider the main result of this paper. Let $\omega^{\nearrow\omega}$ denote the set of all monotonically nondecreasing members of $\omega^\omega$ which limit to infinity. Note that whenever $y$ is a dominating real and $z\in V\cap\omega^{\nearrow\omega}$ then $z\circ y$ and $y\circ z$ are also dominating. The next result shows that when adding a Hechler real this is in some sense the only way to get dominating reals. 

\begin{thm} \label{DomChar}
Let $d$ be a $\Dnd$-generic real over $V$ and let $y\in\omega^\omega\cap V[d]$ be dominating. Then there are $z_0,z_1\in V\cap\omega^{\nearrow\omega}$ so that $y$ eventually dominates $z_0\circ d\circ z_1$.
\end{thm}

The paper is organized as follows. In section 2 we will define the Hechler extensions under consideration and compare them. We will show there that $\D$ and $\Dnd$ are forcing equivalent. We will show that $\D$ and $\Dtree$ are different, although each completely embeds into the other. In section 3 we will focus on $\Dtree$ and prove Theorem \ref{DtreeNotDom}. In section 4 we prove Theorem \ref{DomChar} and use it to obtain Theorem \ref{DDom}. We also give some applications to work of Laflamme \cite{La} and Brendle and L{\"o}we \cite{BrL}. Finally in section 5 we will discuss forcing extensions where no analogue of Theorem \ref{DomChar} holds.


Our notation and terminology is mostly standard. We use $\omega^\omega$ to refer to the set of all functions on the natural numbers, and often we will call elements of $\omega^\omega$ reals. We use $\leq^*$ to refer to the preorder of eventualy domination on $\omega^\omega$. This means that we have $$x\leq^*y \Leftrightarrow (\forall^\infty n) x(n)\leq y(n).$$ A dominating real in a generic extension is a real $y\in\omega^\omega$ for which for all $f\in V\cap\omega^\omega$ we have $f\leq^*y$. An unbounded real in a generic extension is a real $x\in\omega^\omega$ for which for all $f\in V\cap\omega^\omega$ we have $x\not\leq^* f$.

We use $\mathbb{C}$ to signify Cohen forcing, whose conditions we will take to come from either $2^{<\omega}$ or $\omega^{<\omega}$ as the situation demands. When $\mathbb{P}$ is a ccc notion of forcing we abuse notation somewhat and let $V^\mathbb{P}\cap\omega^\omega$ denote the collection of nice names for reals. For two forcing notions $\mathbb{P}$ and $\mathbb{Q}$ we will use $\mathbb{P}\equiv\mathbb{Q}$ to denote forcing equivalence which means 1) for any $G$ a $\mathbb{P}$-generic filter over $V$ there is some $H\in V[G]$ which is a $\mathbb{Q}$-generic filter over $V$ and for which $V[G]=V[H]$ and 2) vice versa: for any $H$ a $\mathbb{Q}$-generic filter over $V$ there is some $G\in V[H]$ which is a $\mathbb{P}$-generic filter over $V$ and for which $V[G]=V[H]$.

\section{Notions of Hechler forcing}

In this section we will define and compare the three variations of Hechler forcing under consideration. All three are $\sigma$-centered partial orderings adding a dominating real, and each consists of two parts: a stem giving a finite approximation of the real being added, and a commitment restricting the possible values the real may take beyond the stem.

The original Hechler forcing $\D$ was introduced by Hechler \cite{H}. In that paper Hechler used nonlinear iterations of $\D$ to prove that for any $\sigma$-directed partially ordered set $P$ there is a generic extension in which $P$ is isomorphic to a cofinal subset of $(\omega^\omega,\leq^*)$. Conditions in $\D$ are pairs $\langle s,f\rangle$ where $s\in\omega^{<\omega}$ and $f\in\omega^\omega$. The ordering is given by $$\langle s',f'\rangle\leq\langle s,f\rangle \Leftrightarrow s\subseteq s', (\forall n)f(n)\leq f'(n)\textrm{ and }(\forall n\in |s'|\setminus |s|)f(n)\leq s'(n).$$

The nondecreasing Hechler forcing $\Dnd$ is the same as $\D$ except that we insist that $s$ be monotonically nondecreasing. This slight tweaking of Hechler forcing was first used by Baumgartner and Dordal in \cite{BD} where among other things they showed that by iterating $\Dnd$ over a model CH one obtains a model where the splitting number $\mathfrak{s}$ is strictly less than the bounding number $\mathfrak{b}$.

The tree Hechler forcing $\Dtree$ is a special case of the forcings made up of trees branching into a filter that were considered by Groszek \cite{Gr}. The forcing $\Dtree$ was first explicitly used by Brendle and L{\"o}we \cite{BrL} to obtain by iteration a model where $\Delta^1_2(\D)$ holds and $\Delta^1_2(\mathbb{E})$ fails. Conditions in $\Dtree$ are trees $T\subseteq\omega^{<\omega}$ with a distinguished stem $s=\stem(T)$ so that for all $t$ in $T$ either $s$ extends $t$ or $t$ extends $s$ and so that whenever $t$ in $T$ extends $s$ we have $(\forall^\infty n)t\smallfrown n\in T$. The forcing is ordered by inclusion: $T'\leq T$ exactly when $T'\subseteq T$.

Though the difference in the definitions of $\D$ and $\Dnd$ is slight and one often appears in arguments where the other would serve just as well, the two have occasionally been treated as separate entities, as in \cite{BrJS}. Intuitively there should be little difference but whether the two are actually equivalent appears to have been an open question. See for example the discussion after definition 3.1.9 in \cite{BaJ}.

The two forcing extensions are in fact the same. This theorem is joint with Itay Neeman.

\begin{thm} \label{HechlerEquiv}
$\D\equiv\Dnd$.
\end{thm}

\begin{proof}
The proof comes in two steps. First we will prove that $\Dnd*\mathbb{C}\equiv\D$, and then we will prove that $\Dnd*\mathbb{C}\equiv\Dnd$.

Suppose $d$ is a $\mathbb{D}$-generic real over $V$. Define the real $\dnd$ by $$\dnd(n)=\min\{d(k):k\geq n\}.$$ Then $\dnd$ is a $\Dnd$-generic real over $V$. Let $d'=d-\dnd$. Now while $d'$ is a Cohen real over $V$ it is not quite true that it is Cohen over $V[\dnd]$. This is because whenever $\dnd(n)\not=\dnd(n+1)$ we have $d(n)=\dnd(n)$. But this is the only barrier. Let $A$ be the set $\{n:\dnd(n)=\dnd(n+1)\}$. Then $d'\upharpoonright A$ is a Cohen real over $V[d]$ (where for Cohen forcing we use the forcing consisting of sequences of natural numbers with domain a finite subset of $A$). Furthermore $V[d]=V[\dnd][d'\upharpoonright A]$.

Going the other way, suppose $d_0$ is a $\Dnd$-generic real over $V$. Let $A$ be the set $\{n:d_0(n)=d_0(n+1)\}$ and suppose $c$ is generic over $V[d_0]$ for the forcing consisting of sequences of natural numbers with domain a finite subset of $A$. Letting $c_0$ agree with $c$ on $A$ and take the value $0$ outside of $A$, we have that $d=d_0+c_0$ is a $\D$-generic real. Since $d_0=\dnd$ we have $V[d]=V[d_0][c]$. Thus $\Dnd*\mathbb{C}\equiv\D$. 


It remains to show $\Dnd*\mathbb{C}\equiv\Dnd$. Towards that end suppose that $d$ is a $\Dnd$-generic real over $V$. Let $\{r_k:k\in\omega\}\subseteq\omega$ enumerate the range of $d$ in increasing order. Let $I_k(d)$ be the interval on which $d$ takes value $r_k$. Let $c\in 2^\omega$ be defined so that $c(k)$ is equal to the parity of the length of the interval $I_k(d)$. We define $d_0$ to be the nondecreasing real with the same range as $d$ but for which $I_k(d_0)$ has half the length (rounded up) of $I_k(d)$. Then it is straightforward to check that $d_0$ is a $\Dnd$-generic real and that $c$ is a Cohen real over $V[d_0]$. Also $V[d]=V[d_0][c]$.

This process is reversible. Given a $\Dnd$-generic real $d_0$ and a Cohen real $c\in 2^\omega$ let $d$ be the nondecreasing real with the same range as $d_0$ and for which the length of $I_k(d)$ is equal to $c(k)$ plus twice the length of $I_k(d_0)$. Then $d$ is $\D$-generic over $V$ and $V[d]=V[d_0][c]$. This completes the proof.
\end{proof}

Now we compare the forcings $\D$ and $\Dtree$. The next proposition shows that each is a subforcing of the other.

\begin{prop} \label{mutAdd}
Forcing with $\Dtree$ adds a $\D$-generic real, and forcing with $\D$ adds a $\Dtree$-generic real.
\end{prop}

\begin{proof}
That forcing with $\D$ adds a $\Dtree$-generic real was observed by Brendle and L{\"owe} in \cite{BrL}. Given $d$ a $\mathbb{D}$-generic let $N\in\omega$ be such that $N\leq n$ implies $n<d(n)$. Define $d'$ by $d'(n)=d(n)$ for $n\leq N$, and recursively $d'(n+1)=d(d'(n))$ for $N\leq n$. Then $d'$ is a $\Dtree$-generic real over $V$.

For the other direction let $d$ be a $\Dtree$-generic real over $V$. Take $d'$ to be defined by letting $d'(n)$ take the value of half that of $d(n)$, rounded down. It is not difficult to check that $d'$ is also a tree Hechler real over $V$. Now define $c\in 2^\omega$ by setting $c(n)$ equal to the parity of $d(n)$. Then $c$ is Cohen over $V[d']$. A theorem of Truss \cite{Tr} says that given $d'$ any dominating real over V and $c'\in\omega^\omega$ any Cohen real over $V[d']$ one has that $d'+c'$ is a $\mathbb{D}$-generic real over $V$. This completes the proof.
\end{proof}

We will show that $\D$ and $\Dtree$ are not forcing equivalent, despite the fact that each of the two forcings adds a generic real for the other. There appears to be no other example of the failure of the natural Cantor-Bernstein theorem for forcing notions in the literature. After finding this result the present author asked on Mathoverflow whether such examples had previously been known. There, based on a conversation with Arthur Apter, Joel David Hamkins produced another example. He showed that if one takes $\mathbb{P}$ to be the forcing to add a Cohen subset of $\omega_2$ and $\mathbb{S}$ to be the forcing to add a stationary nonreflecting subset of $\omega_2$, then together $\mathbb{P}$ and $\mathbb{P}*\mathbb{S}$ give such an example. The reader may find more details at \cite{MO}.

We now give some notation and terminology for stems consistent with that introduced in \cite{BrL}. We will be using the same terminology for $\D$ and $\Dtree$; which forcing notion we mean will be clear from context.

First we consider $\D$ (and $\Dnd$). For a condition $p=\langle s,f\rangle$ and $t\in\omega^{<\omega}$ we write $t\leq p$ to mean $$s\subseteq t\textrm{ and }(\forall n\in |t|\setminus |s|)t(n)\geq f(n).$$ We say that $s\in\omega^{<\omega}$ forces a formula $\varphi$ if there exists some commitment $f$ for which $\langle s,f\rangle\Vdash\phi$. Let $A\subseteq\omega^{<\omega}$. We will say that $s$ favors $A$ if for every choice of commitment $f$ there is some $t\in A$ so that $t\leq\langle s,f\rangle$. We say that $s$ favors $\varphi$ if $s$ favors the set $\{t\in\omega^{<\omega}:t\textrm{ forces }\varphi\}$. Notice that $s$ favors $\varphi$ exactly when $s$ does not force $\neg\varphi$.

Our terminology for $\Dtree$ is similar. We write $t\leq T$ to mean $\stem(T)\subseteq t$ and $t\in T$. We say $s$ forces $\varphi$ when there is $T\in\Dtree$ with $\stem(T)=s$ and $T\Vdash\varphi$. We say that $s$ favors $A$ if for every $T\in\mathbb{D}$ with $\stem(T)=s$ there is $t\leq T$ with $t\in A$. When $T\in\Dtree$ and $\stem(T)\subseteq t$, write $T_t$ for the tree with $\stem(T_t)=t$ containing exactly the initial segments of $t$ and the extensions of $t$ in $T$.

Since any two conditions with the same stem are compatible any condition with stem forcing $\varphi$ may be strengthened to a condition forcing $\varphi$.

\section{Unbounded and dominating reals in the tree Hechler extension}

Our goal in this section is to prove Theorem \ref{DtreeNotDom}. The following easy proposition characterizing the unbounded reals in a generic extension gives the motivation for our method. We leave the proof to the reader.

\begin{prop}\label{simpleUnboundedchar}
Let $\mathbb{P}$ be an arbitrary notion of forcing, and let $\dot{x}\in V^{\mathbb{P}}\cap\omega^\omega$. Then $$\Vdash_{\mathbb{P}}``\dot{x}\textrm{ is unbounded}" \Longleftrightarrow (\forall p\in\mathbb{P})(\exists^{\infty}n)(\forall i)p\not\Vdash\dot{x}(n)\leq i.$$
\end{prop}


In order to prove Theorem \ref{DtreeNotDom} we give a strengthening of Proposition \ref{simpleUnboundedchar} for the case where $\mathbb{P}=\Dtree$. We give a characterization of the unbounded reals in the tree Hechler extension expressed using stems rather than outright conditions.

\begin{lemma}\label{mainTreeLemma}
Fix $\dot{x}\in V^{\Dtree}$. Set $A=\{t\in\omega^{\omega}|(\exists n\geq|t|)(\forall i)t\textrm{ favors }i<\dot{x}(n)\}$. Then $$\Vdash_{\Dtree}``\dot{x}\textrm{ is unbounded}"\Longleftrightarrow\textrm{ every $s\in\omega^{<\omega}$ favors $A$.}$$
\end{lemma}

\begin{proof}
First we go from right to left. Let $z$ be a real in the ground model. Suppose for contradiction that there is some $T\Vdash_{\Dtree}(\forall n\geq N)\dot{x}(n)\leq z(n)$. By strengthening $T$ as necessary we may assume that $s=\stem(T)$ has length greater than $N$. Since $s$ favors $A$ by further strengthening $T$ if necessary we may assume that $s$ belongs to $A$. But now there is some $n\geq|s|\geq N$ so that $(\forall i)$ $s$ favors $i<\dot{x}(n)$. Take $i=z(n)$. We may extend $T$ to $T'$ with $\stem(T')$ forcing $z(n)<\dot{x}(n)$. That is a contradiction.

The left to right implication is more involved. We argue by contrapositive. Suppose there is some $s$ which does not favor $A$. Then we can find a tree $T$ with $\stem(T)=s$ for which $t\leq T$ implies $t\not\in A$. To simplify notation we will assume that $\stem(T)=\varnothing$ and $T=\omega^{<\omega}$; the simplification does little to change the argument.

Now by assumption, every $s\in\omega^{<\omega}$ fails to belong to $A$. That means there is a function $v:\omega^{<\omega}\times\omega\rightarrow\omega$, and for every $s$ and $n$ with $n\geq|s|$ some tree $T^{s,n}$ with $\stem(T^{s,n})=s$ such that $$T^{s,n}\Vdash\dot{x}(n)\leq v(s,n).$$

\medskip

\noindent\emph{Claim 1.} There exists $U\in\Dtree$ with $\stem(U)=\varnothing$ such that $$(\forall s\leq U)(\forall n\geq|s|)(\forall^{\infty}m)U_{s\smallfrown m}\subseteq T^{s,n}.$$
\begin{proof}[Proof of Claim 1.]
A fusion argument. We define a sequence of trees $\{U^l|l\in\omega\}$ such that
\begin{enumerate}
\item $\stem(U^l)=\varnothing$, $U^{l+1}\subseteq U^l$
\item $l<j$, $s\in U^l$ with $|s|\leq l$ implies $s\in U^j$
\item for all $s\in U^{l+1}$ with $|s|=l$ we have $(\forall n)(\forall^\infty m)U^{l+1}_{s\smallfrown m}\subseteq T^{s,n}$.
\end{enumerate} Then we can take $U=\cap_{l<\omega}U^l$. Start with $U^0=\omega^{<\omega}$. Supposing $U^l$ is defined, let $s\in U^l$ with $|s|=l$. For each $n\geq |s|$ there is some $i(n)$ so that $m\geq i(n)$ implies $s\smallfrown m\in T^{s,n}$. Then define $U^{l+1}_{s\smallfrown m}$ to be the intersection of $U^l_{s\smallfrown m}$ with each $T^{s,n}$ for $n,i(n)\leq m$. Then $m\geq i(n),n$ will imply $U^{l+1}_{s\smallfrown m}\subseteq T^{s,n}$.
\end{proof}

Now fix $U$ as in the claim and let $c:\omega^{<\omega}\times\omega\rightarrow\omega$ be such that $$(\forall s\leq U)(\forall n\geq|s|)(\forall m\geq c(s,n))U_{s\smallfrown m}\subseteq T^{s,n}.$$ By further extending $U$ we may assume that for every $s\in\omega^{<\omega}$, whenever $n\leq\max(\ran(s))$ we have $U_s\subseteq T^{s,n}$. 

Define $f\in\omega^\omega$ so that whenever $|s|,\ran(s)\leq n$ and $m<c(s,n)$ we have $v(s\smallfrown m,n)\leq f(n)$. Let $g\in\omega^\omega$ be such that $v(s,n)\leq g(n)$ whenever $|s|,\max(\ran(s))\leq n$. We claim $U\Vdash\dot{x}\leq^*\max(f,g)$.

We work now in an arbitrary generic extension $V[G]$ with $U\in G$. Let $d$ be the corresponding tree Hechler real. Then $G$ is exactly the set of members of $\Dtree$ through which $d$ is a branch. Let $x$ be the evaluation of $\dot{x}$ via $G$. Let $n\in\omega$ with $d(k)<n\leq d(k+1)$. For sufficiently large $k$ we have $k\leq d(k)$ so by taking $n$ sufficiently large we may assume that $k<n$. We show $x(n)\leq \max\{f(n),g(n)\}$.

\medskip

\noindent\emph{Claim 2.} $x(n)\leq v(d\upharpoonright k+2,n)$.
\begin{proof}[Proof of Claim 2.]
This is because $T^{d\upharpoonright k+2,n}$ belongs to $G$, which follows from our assumption that $U_s\subseteq T^{s,n}$ whenever $n\leq\max(\ran(s))$.
\end{proof}

Now we split into two cases. In the first case, if $T^{d\upharpoonright k+1,n}$ belongs to $G$ then $x(n)\leq v(d\upharpoonright k+1,n)\leq g(n)$. In the second case, if $T^{d\upharpoonright k+1,n}\not\in G$ then we claim that $v(d\upharpoonright k+2,n)\leq f(n)$ which by Claim 2 will give $x(n)\leq f(n)$. Since $T^{d\upharpoonright k+1,n}\not\in G$ it follows that $U_{d\upharpoonright k+2}\not\subseteq T^{d\upharpoonright k+1,n}$ (because $U_{d\upharpoonright k+2}\in G$). It follows by definition of $c$ that $d(k+1)<c(d\upharpoonright k+1,n)$. Then the definition of $f$ gives $v(d\upharpoonright k+2,n)\leq f(n)$ as required.
\end{proof}

Armed with Lemma \ref{mainTreeLemma} we can prove Theorem \ref{DtreeNotDom}.

\begin{proof}[Proof of Theorem \ref{DtreeNotDom}]
Fix $\dot{x}\in V^{\Dtree}\cap\omega^\omega$ with $\Vdash\textrm{``$\dot{x}$ is unbounded"}$. Taking $A$ as in Lemma \ref{mainTreeLemma} we know that every $s\in\omega^{<\omega}$ favors $A$. Let $\phi:A\rightarrow\omega$ satisfy $\phi(t)\geq|t|$ and $$(\forall i)t\textrm{ favors }i<\dot{x}(\phi(t)).$$ We let $d$ be a tree Hechler real over $V$ and work in $V[d]$.

Define $d'$ by \begin{displaymath}
   d'(k) = \left\{
     \begin{array}{lr}
       d(n) & \textrm{ where $n$ is least such that $k=\phi(d\upharpoonright n)$, if such an $n$ exists}\\
       d(k) & \textrm{ if no such $n$ exists}
     \end{array}
   \right.
\end{displaymath} 

\medskip

\noindent\emph{Claim 1.} $\Vdash\textrm{$d'$ is dominating}$.
\begin{proof}[Proof of Claim 1.]
For any ground model real $f$ and any $T\in\mathbb{D}$ we can extend to $T'$ with $\stem(T)=\stem(T')$ such that $\stem(T')\subseteq s$ and $s\smallfrown m\in T'$ implies $m\geq f(\phi(s))$.
\end{proof}

Now observe that if $s$ favors $\varphi$ then $(\exists^\infty m)s\smallfrown m$ favors $\varphi$. This allows us to define $z\in V\cap\omega^{\nearrow\omega}$ such that $$(\forall s\in A)(\exists^\infty m)s\smallfrown m\textrm{ favors }\dot{x}(\phi(s))\geq z(m).$$
Then, because $z$ belongs to the ground model it follows that $z\circ d'$ is dominating. Thus the theorem will be proved given the following claim.

\medskip

\noindent\emph{Claim 2.}
$\Vdash(\exists^\infty k)z(d'(k))\leq \dot{x}(k)$.

\begin{proof}[Proof of Claim 2.]
Fix $N$ and $T$. We want to find $k\geq N$ and $U\leq T$ such that $U\Vdash z(d'(k))\leq\dot{x}(k)$. Let $s=\stem(T)$. We may assume that $|s|\geq N$ and that $j\geq i\geq |s|$ implies $d(i)\leq d(j)$. Since $s$ favors $A$ we may also assume that $s\in A$. Now pick $m$ such that $s\smallfrown m\in T$ and $s\smallfrown m$ favors $\dot{x}(\phi(s))\geq z(m)$. Since $s\smallfrown m\in T$  there is some $U\leq T$ such that $\stem(U)=s\smallfrown m$ and also $$U\Vdash\dot{x}(\phi(s)))\geq z(m).$$ Now taking $l=|s|$ we have $$U\Vdash\dot{x}(\phi(d\upharpoonright l))=\dot{x}(\phi(s))\geq z(m)=z(d(l))\geq z(d'(\phi(d\upharpoonright l)).$$ And $\phi(d\upharpoonright l)\geq l\geq N$. So $k=\phi(d\upharpoonright l)$ satisfies the claim.
\end{proof}

\end{proof}

\section{Unbounded and dominating reals in the standard Hechler extension}


\subsection{Proof of Theorem 3}

Our objective in this section is to prove Theorem \ref{DomChar}. Let us note that although we have seen that $\D$ and $\Dnd$ are equivalent as forcing notions, nonetheless the direct analogue of Theorem \ref{DomChar} for $\D$ is not true. For example, suppose that $d$ is a $\D$-generic real and let $d_0\in V[d]\cap\omega^\omega$ satisfy $$(\forall n) d_0(2n)=d_0(2n+1)=\min\{d(2n),d(2n+1)\}.$$ Then $d_0$ is a dominating real but for any $z_0,z_1\in V\cap\omega^{\nearrow\omega}$ we have $z_0\circ d\circ z_1\not\leq^*d_0$.

Therefore we will exclusively be working with the poset $\Dnd$ and thus we will only be concerned with stems $s$ which are nondecreasing. For the rest of this section when we refer to finite sequences of naturals we shall always mean nondecreasing ones, even when not explicitly stated. Let $\omega^{\nearrow<\omega}$ be the collection of such sequences, and let $\omega^{\nearrow m}$ be the collection of nondecreasing sequences of naturals of length $m$.

To motivate we start with the following simple proposition about dominating reals in $V^\Dnd$.

\begin{prop}\label{DomFact}
Let $\dot{y}\in V^\mathbb{\Dnd}\cap\omega^\omega$ and let $A=\{t|(\forall^\infty n)(\forall i)\textrm{$t$ forces $i\leq\dot{y}(n)$}\}$. Then $$\Vdash_{\mathbb{\Dnd}}``\dot{y}\textrm{ is dominating }"\Longrightarrow \textrm{ every $s$ favors $A$}.$$
\end{prop}

\begin{proof}
Argue by contrapositive; if some $s$ does not favor $A$ then we can find some $f\in\omega^\omega$ such that $t\leq\langle s,f\rangle$ implies $t\not\in A$. For each such $t$ we have $(\exists^\infty n)(\exists i)t$ favors $\dot{y}(n)<i$. This allows us to define a function $z\in\omega^\omega$ so that for each $t\not\in A$ we have $(\exists^\infty n)\textrm{$t$ favors }\dot{y}(n)<z(n)$. Thus $$\langle s,f\rangle\Vdash(\exists^\infty n)\dot{y}(n)<z(n).$$
\end{proof}

For the rest of this section we let $\dot{y}\in V^{\Dnd}\cap\omega^\omega$ and take $A$ to be defined as in Proposition \ref{DomFact}. Let $\phi:A\rightarrow\omega$ be defined so that that $\phi(s)$ equals the least $N$ such that $$(\forall n\geq N)(\forall i) s\textrm{ forces }i\leq\dot{y}(n).$$ We extend $\phi$ to a function $\phi:\omega^{\nearrow<\omega}\rightarrow\omega\cup\{\infty\}$ by letting $\phi(s)=\infty$ when no such $N$ exists.

Our strategy for characterizing when $\dot{y}$ is a dominating real is to analyze the growth of the function $\phi$. Supposing for example that $\dot{y}$ were of the form $z_0\circ\dot{d}\circ z_1$ for some $z_0,z_1\in \omega^{\nearrow\omega}$, it is not hard to see we would have that $\phi(s)$ is a function of the \emph{length} of $s$. It turns out that this is essentially an exact characterization of the dominating reals.

\begin{defn}
Fix $q\in\mathbb{D}$. We say that $\phi$ is \emph{length bounded below $q$} if there is some function $\psi\in\omega^\omega$ so that whenever $s\leq q$ we have $\phi(s)\leq\psi(|s|)$.
\end{defn}

We are now ready to give several characterizations of the dominating reals in $V^\Dnd$. Let $B\subseteq\omega^{\nearrow<\omega}$ be the collection $$\{s|(\exists m)(\exists\{t_l:l\in\omega\}\subseteq\omega^{\nearrow m})\lim_{l<\omega}t_l(0)=\infty\textrm{ and }\lim_{l<\omega}\phi(s\smallfrown t_l)=\infty\}.$$ The definition of $B$ is motivated in part by the Baumgartner-Dordal rank analysis of $\Dnd$. For someone hoping that $\phi$ is everywhere length bounded $B$ is a bad set and in order for $\dot{y}$ to be a dominating real we must mostly be able to avoid it. 

\begin{lemma} \label{MainLemma}
The following are equivalent:
\begin{enumerate}
\item $\Vdash``\dot{y}\textrm{ is dominating"}$
\item $(\forall p)(\exists q\leq p)(\forall t\leq q)t\not\in B$.
\item $(\forall p)(\exists q\leq p)\textrm{ $\phi$ is length bounded below $q$}.$
\item $(\forall p)(\exists q\leq p)(\exists z_0,z_1\in\omega^{\nearrow\omega})q\Vdash z_0\circ\dot{d}\circ z_1\leq^*\dot{y}$.
\end{enumerate}
\end{lemma}

Notice that (1) implies (4) gives Theorem \ref{DomChar}.

\begin{proof}

That (4) implies (1) is clear.

We show (1) implies (2). For each $s\in B$ fix a witnessing sequence $\{t_l^s:l\in\omega\}$. Then we may define a function $z\in\omega^\omega$ such that $$(\forall s\in B)(\forall N)(\exists n,l>N)s\smallfrown t_l^s\textrm{ favors }\dot{y}(n)<z(n).$$ Suppose now that (2) failed and there was some $p$ so that $(\forall q\leq p)(\exists s\leq q)s\in B$. We claim that $$p\Vdash(\exists^\infty n)\dot{y}(n)< z(n).$$

If not then there is some $q\leq p$ with $q\Vdash(\forall n\geq N_0)z(n)\leq\dot{y}(n)$. Write $q=\langle t,f\rangle$. There is $s\in B$ with $s\leq q$. Since $s\in B$ we may take $l,n\in\omega$ so that $s\smallfrown t_l^s\textrm{ favors }\dot{y}(n)<z(n)$ and $l,n$ are large enough that $n\geq N_0$, $s\smallfrown t_l^s\leq q$. Since $s\smallfrown t_l^s$ favors $\dot{y}(n)<z(n)$ we may further extend $q$ to force $\dot{y}(n)<z(n)$, a contradiction.

Next we show that (2) implies (3). Fix a condition $p\in\Dnd$. Taking $q\leq p$ as given by (2), write $q=\langle s,f\rangle$. We will define an $r\leq q$ so that $\phi$ is length bounded below $r$. In particular we construct functions $\psi$, $f'\in\omega^\omega$ such that $s\smallfrown t\leq\langle s,\max\{f,f'\}\rangle$ implies $\phi(s\smallfrown t)\leq\psi(|s\smallfrown t|)$. Start by setting $\psi(|s|)$ equal to $\phi(s)$. Before we proceed further let us note that when $t\not\in B$ it follows that for every $m$ there is some $N,L$ so that if $t\in\omega^{\nearrow m}$ with $t(0)\geq L$ then $\phi(s\smallfrown t)\leq N$.

Fix $m\in\omega$. We define $\psi(|s|+m+1)$, $f'(|s|+m)$. To do so we recursively define a finite set $S_m\subseteq\omega^{\nearrow\leq m}$, and we simultaneously define $L_t,N_t\in\omega$ for each $t\in S_m$. We will make sure that $t\in S_m$ implies $s\smallfrown t\not\in B$. Start by placing $\varnothing\in S_m$. Now suppose that $t\in S_m$. Since $s\smallfrown t\not\in B$ there is $L_t,N_t\in\omega$ such that whenever $u\in\omega^{\nearrow<\omega}$ with $|u|=m+1-|t|$ and $u(0)\geq L_t$, $\phi(s\smallfrown t\smallfrown u)\leq N_t$. If $|t|<m$ put $s\smallfrown t\smallfrown i\in S_m$ whenever $i<L_t$ and $s\smallfrown t\smallfrown i\not\in B$. That completes our definition of $S_m$. Let $\psi(|s|+m+1)=\max_{t\in S_m}N_t$ and $f'(|s|+m)=\max_{t\in S_m}L_t$.

Let us check that this works. Suppose $s\smallfrown t\leq\langle s,\max\{f,f'\}\rangle$, and say $|t|=m+1$. Notice that $t\upharpoonright0=\varnothing\in S_m$. Take $k$ as large as possible with $t\upharpoonright k\in S_m$. First suppose $k<m$. Then since $s\smallfrown t\upharpoonright k+1\not\in B$ by definition of $S_m$ we must have $t(k)\geq L_{t\upharpoonright k}$ which implies that $\phi(s\smallfrown t\upharpoonright k\smallfrown t\upharpoonright[k+1,m])\leq N_t\leq\psi(|s|+m+1)$. Now suppose $k=m$. Since $t(m)\geq f'(|s|+m)\geq L_{t\upharpoonright k}$ we have $$\phi(s\smallfrown t\upharpoonright k\smallfrown t(m))\leq N_{t\upharpoonright k}\leq\psi(|s|+m+1)$$ as needed.

Finally we show that (3) implies (4). Fix $p\in\Dnd$ and let $q\leq p$ with $\phi$ length bounded below $q$. Let $\psi\in\omega^\omega$ witness the bound. We may assume without loss of generality that $\psi$ is a strictly increasing function. Whenever $t\leq q$ and $n\geq\psi(|t|)$ we have for every $i\in\omega$ some commitment $f^t_{n,i}$ such that $$\langle t,f^t_{n,i}\rangle\Vdash i\leq\dot{y}(n).$$

Now say $q=\langle s,f\rangle$. Our goal is to construct $z_0,z_1$ and $h$ so that
$$(*)\textrm{ }\langle s,\max\{f,h\}\rangle\Vdash(\forall^\infty n) z_0(\dot{d}(z_1(n)))\leq\dot{y}(n).$$ We let $z_1\in\omega^\omega$ be defined by having $z_1(n)=l$ whenever $\psi(l)\leq n<\psi(l+1)$. To define $h$ and $z_0$ we will make use of the following simple proposition whose proof we leave to the reader.

\begin{prop} \label{mainProp}
Let $\mathcal{G}$ be a countable subset of $\omega^\omega$. Then there is a $z\in\omega^{\nearrow\omega}$ so that for all $g\in\mathcal{G}$ we have $$(\forall^\infty m)g(z(m))\leq m.$$
\end{prop} 


Fix $n,j\in\omega$ with $|s|\leq l$ where $l=z_1(n)$. We define a finite set $S_{n}(j)\subseteq\omega^{\nearrow\leq l}$ by recursion. We will guarantee that $t\in S_{n}(j)$ implies $t\leq q$. In particular $f^t_{n,j}$ will be defined for $t\in S_{n}(j)$. Start by putting $s$ in $S_{n}(j)$. Then, whenever $t\in S_{n}(j)$ place $u$ in $S_{n}(j)$ if $u\leq q$, $t\subseteq u$, $|u|\leq l$ and $u(|u|-1)<f_{n,j}^t(|u|-1)$. Since we have restricted our attention to nondecreasing sequences there are only finitely many options for $u$. Now define $g_{n,k}$ by $$g_{n,k}(j)=\max\{f^t_{n,j}(k):t\in S_{n}(j)\}.$$ Let $\mathcal{G}$ be the collection $$\{g_{n,k}:z_1(n)\leq k\}.$$ Apply Proposition \ref{mainProp} to $\mathcal{G}$ to obtain $z_0$. By the defining property of $z_0$ for each $k$ the set $$X_{n,k}=\{m:(\exists t\in S_n(z_0(m)) m<f_{n,z_0(m)}^t(k)\}$$ is finite. Let $h\in\omega^\omega$ with $f_{n,z_0(m)}^t(k)\leq h(k)$ whenever $m\in X_k$, $z_1(n)\leq k$ and $t\in S_n(z_0(m))$. Then $h$ satisfies $$(\dagger)\hspace{.1cm}(\forall t\in S_{n}(z_0(m))) m<f^t_{n,z_0(m)}(k) \Rightarrow f^t_{n,z_0(m)}(k)\leq h(k)$$ whenever $z_1(n)\leq k$.

We complete the proof by checking that $(*)$ holds. Let $d$ be a $\Dnd$-generic real so that $\langle s,\max\{f,h\}\rangle$ belongs to the corresponding generic filter $G$. Fix $n\geq\psi(|s|)$ and let $l=z_1(n)$. 

\medskip

\noindent\emph{Claim 1.} For $k\geq l,|s|$ and $t\in S_{n}(z_0(d(l)))$ we have $f^t_{n,z_0(d(l))}(k)\leq d(k)$. 

\begin{proof}[Proof of Claim 1.]
We split into two cases. In the first case if $f^t_{n,z_0(d(l))}(k)\leq d(l)$ then we are done since $l\leq k$ and $d$is  nondecreasing. In the second case $d(l)<f^t_{n,z_0(d(l))}(k)$. But then by $(\dagger)$ we have $$f^t_{n,z_0(d(l))}(k)\leq h(k)\leq d(k).$$ We have $h(k)\leq d(k)$ since $|s|\leq k$ and $\langle s,h\rangle$ belongs to $G$.
\end{proof}

Now take $l_0\leq l$ to be as large as possible so that $d\upharpoonright l_0$ belongs to $S_{n}(z_0(d(l))$.

\medskip

\noindent\emph{Claim 2.} For $k\geq l_0$ we have $f^{d\upharpoonright l_0}_{n,z_0(d(l))}(k)\leq d(k)$.

\begin{proof}[Proof of Claim 2.]
If not there is some violating $k\geq l_0$. By Claim 1 we know $k<l$. We have that $q$ belongs to $G$ and so $d\upharpoonright k+1\leq q$. We also have $d(k)<f^{d\upharpoonright l_0}_{n,z_0(d(l))}(k)$. Thus by the definition of $S_{n}(z_0(d(l)))$ we find that $d\upharpoonright k+1\in S_{n}(z_0(d(l)))$ which is contrary to the maximality of $l_0$.
\end{proof}

By Claim 2 (and the fact that $z_1(n)=l$) we have $$\langle d\upharpoonright l_0,f^{d\upharpoonright l_0}_{n,z_0(d(z_1(n)))}\rangle\in G.$$ Since this condition forces that $z_0(d(z_1(n)))\leq\dot{y}(n)$ we are done.
\end{proof}

\subsection{Proof of Theorem 2}

Using Theorem \ref{DomChar} we can now prove Theorem \ref{DDom}.

\begin{proof}
Let $d$ be a $\Dnd$-generic real. Our goal is to to produce an unbounded real $x$ in $V[d]$ which is eventually dominated by every dominating real. Fix $n\in\omega$. Let $k$ be least with $d(k)\geq n$. Then we set $x(n)=i$ where $i$ is large as possible so that $$(\forall j\in[k,k+i])d(k)=d(j).$$ An easy density argument shows that $x$ is indeed unbounded. To show that $x$ is eventually dominated by every dominating real, it is enough by Theorem \ref{DomChar} to show that $x\leq^* z_0\circ d\circ z_1$ for every $z_0,z_1\in V\cap\omega^{\nearrow\omega}$.

Fix such $z_0$ and $z_1$ and let $f\in V\cap\omega^\omega$ satisfy $$(1)\textrm{ }(\forall n) n<f(z_0(n)) \textrm{ and (2) }(\forall n)n<f(z_1(n)).$$ 

We claim then that for any $s\in\omega^{\nearrow<\omega}$ we have $$\langle s,f\rangle\Vdash(\forall^\infty n)x(n)<z_0(d(z_1(n)))$$ which will complete the proof.

Assume instead that $\langle s,f\rangle$ belongs to the generic filter $G$ corresponding to $d$ and yet $z_0(d(z_1(n)))\leq x(n)$ holds for infinitely many $n$. We know $$(\forall n\geq |s|)f(n)\leq d(n).$$ We also know that $z_0(d(z_1(n)))$ is dominating and thus for sufficiently large $n$ we have $z_1(n)\leq z_0(d(z_1(n)))$. Fix an $n$ with $|s|\leq n$, $|s|,z_1(n)\leq z_0(d(z_1(n)))$ and $z_0(d(z_1(n)))\leq x(n)$. Let $k$ be least with $n\leq d(k)$. By (2) $n\leq d(z_1(n))$ and thus $k\leq z_1(n)$. By assumption $x(n)$ is larger than $z_0(d(z_1(n)))$ and by the definition of $x$ we have that $d$ is fixed on the interval $[k,k+x(n)]$ and therefore $$d(k)=d(z_1(n))=d(z_0(d(z_1(n))))=d(x(n)).$$ 

But applying (1) with $d(z_1(n))$ in place of $n$ we also get $$d(z_1(n))<f(z_0(d(z_1(n))))\leq d(z_0(d(z_1(n))))$$ which brings us to a contradiction.
\end{proof}

\subsection{Consequences}

In this subsection we mention some consequences of the other work from this section. Let $d$ be a $\Dnd$-generic real, and let $\mathcal{D}$ be the collection of dominating reals in $V[d]$. 

\begin{cor} \label{IsoStr}
The structures $(V\cap\omega^\omega,\leq^*)$ and $(\mathcal{D},{}^*\hspace{-.13cm}\geq)$ are cofinally isomorphic.
\end{cor}

\begin{proof}
From Theorem \ref{DomChar} we have that the set $\{z\circ d\circ z:z\in V\cap\omega^{\nearrow\omega}\}$ is cofinal in $(\mathcal{D},{}^*\hspace{-.13cm}\geq)$.  In $V$ there is a cofinal mapping $z\mapsto z'$ from $\omega^\omega$ to $\omega^{\nearrow\omega}$ such that $$z_0\leq^*z_1\Leftrightarrow z_0' {}^*\hspace{-.13cm}\geq z_1'.$$ For $z_0,z_1\in V\cap\omega^{\nearrow\omega}$ we also have $$z_0\leq^*z_1\Leftrightarrow z_0\circ d\circ z_0\leq^*z_1\circ d\circ z_1.$$ (The right to left direction uses the genericity of $d$). The corollary follows.
\end{proof}

An interesting and immediate consequence of Corollary \ref{IsoStr} is the following.

\begin{cor} \label{ItaysFav}
Let $\{d_n:n\in\omega\}\in V[d]$ be a countable collection of dominating reals. Then there is a single dominating real $d^*$ such that $d^*\leq^* d_n$ for every $n\in\omega$.
\end{cor}

In the terminology of Laflamme \cite{La} Corollary \ref{ItaysFav} says that $V\cap\omega^\omega$ has uncountable upperbound. In the cited paper Laflamme makes the following definitions. Let $\mathcal{F}\subseteq\omega^\omega$ be a bounded family of functions. Then $\mathcal{F}^\downarrow\subseteq\omega^\omega$ is the set of functions dominating $\mathcal{F}$. (So if $\mathcal{F}=V\cap\omega^\omega$ then $\mathcal{F}^\downarrow=\mathcal{D}$.)
$$\mathfrak{b}(\mathcal{F})=\min\{\mathcal{|H|}:\mathcal{H}\subseteq\mathcal{F}\textrm{ is unbounded in }\mathcal{F}\}$$
$$\mathfrak{d}(\mathcal{F})=\min\{\mathcal{|H|}:\mathcal{H}\subseteq\mathcal{F}\textrm{ is dominating in }\mathcal{F}\}$$
$$\mathfrak{b}^\downarrow(\mathcal{F})=\min\{\mathcal{|H|}:\mathcal{H}\subseteq\mathcal{F}^\downarrow\textrm{ is unbounded in }(\mathcal{F}^\downarrow,{}^*\hspace{-.13cm}\geq)\}$$

We see then that working in $V[d]$ we have that $\mathfrak{b}(V\cap\omega^\omega)=\mathfrak{b}^V$, $\mathfrak{d}(V\cap\omega^\omega)=\mathfrak{d}^V$ and $\mathfrak{b}^\downarrow(V\cap\omega^\omega)=\mathfrak{b}^V$. In section 4 of his paper Laflamme constructed several ZFC examples of bounded families of $\mathcal{F}$ to achieve various values of $\mathfrak{b}(\mathcal{F})$, $\mathfrak{d}(\mathcal{F})$ and $\mathfrak{b}^\downarrow(\mathcal{F})$. In each of his constructions one of these three parameters is countable. Since for any regular uncountable cardinals $\kappa\leq\lambda$ one may find a ground model $V$ with $\mathfrak{b}^V=\kappa$ and $\mathfrak{d}^V=\lambda$, our corollary gives for any such $\kappa$, $\lambda$ the consistency of the existence of a bounded family $\mathcal{F}$ with $\mathfrak{b}(\mathcal{F})=\kappa$, $\mathfrak{d}(\mathcal{F})=\lambda$ and $\mathfrak{b}^\downarrow(\mathcal{F})=\kappa$. Laflamme also specifically asked whether one could consistently obtain a family with $\mathfrak{b}^\downarrow(\mathcal{F})=\mathfrak{b}$ and $\mathfrak{b}(\mathcal{F})=\mathfrak{b}$. He showed that consistently there is no such family. Since $V[d]$ satisfies $\mathfrak{b}=\omega_1$ (\cite{BrJS}), by starting with a model $V$ with $\mathfrak{b}^V=\omega_1$ we find that in $V[d]$ such a family does exist.

Now we turn to some recent work of Brendle and L{\"o}we. In their paper \cite{BrL} the authors were concerned with building models containing many Hechler generic reals but no eventually different reals. One consequence of their work is the following dichotomy theorem for reals in $\Dtree$. The authors originally stated their result as holding for `Hechler reals', a term the authors use as a catch-all. 
The proof they gave was for $\Dtree$.


\begin{thm}[Brendle and L{\"o}we] \label{BLthm}
Let $d$ be a $\Dtree$-generic real and let $x\in V[d]\cap\omega^\omega$. Then

\begin{enumerate}
\item either $x$ is dominating, or
\item $x$ is not eventually different over $V$ (that is, there is some $f\in V\cap\omega^\omega$ such that $(\exists^\infty n)f(n)=x(n))$.
\end{enumerate}
\end{thm}

Using characterization (2) from Lemma \ref{MainLemma} we get the same dichotomy for reals in $\D$.

\begin{cor}
Let $d$ be a $\D$-generic real and let $x\in V[d]\cap\omega^\omega$. Then

\begin{enumerate}
\item either $x$ is dominating, or
\item or $x$ is not eventually different over $V$.
\end{enumerate}
\end{cor}

\begin{proof}
By Theorem \ref{HechlerEquiv} we may work with $d$ a $\Dnd$-generic real instead. Let $\dot{x}\in V^\D\cap\omega^\omega$ and suppose $$p\Vdash``\dot{x}\textrm{ is not dominating"}.$$ Using a version of Lemma \ref{MainLemma} relativized to $\Dnd$ restricted to conditions below $p$, there is $q\leq p$ such that $$(\forall r\leq q)(\exists s\leq r)s\in B.$$ For each $s\in B$ let $\{t_l^s:l\in\omega\}$ be a witnessing sequence. 

Now notice that $$(\forall i)t\textrm{ forces }\dot{x}(n)\geq i\Leftrightarrow (\forall i)t\textrm{ forces }\dot{x}(n)\not=i.$$ Thus if $\phi(t)>N$ that means there exists some $n\geq N$ and $i$ such that $t$ favors $\dot{x}(n)=i$. So we may define a function $y\in\omega^\omega$ so that $$(\forall s\in B)(\forall N)(\exists n,l>N)s\smallfrown t_l^s\textrm{ favors }\dot{x}(n)=y(n).$$ Then $$q\Vdash(\exists^\infty n)\dot{x}(n)=y(n).$$
\end{proof}

Also in \cite{BrL} the authors conjectured (Conjecture 15) that given a Hechler real $d$ and a new real $x$ in $V[d]$ either $V[x]$ is equivalent to a Hechler extension of $V$ or $V[x]$ is equivalent to a Cohen extension of $V$. The authors there use the term `Hechler real' as a catch-all, and so their conjecture has several interpretations. Our results show that whether one interprets the term `Hechler real' in their conjecture to mean $\Dtree$-generic real or interprets it to mean $\D$-generic real the conjecture is false. This is because (by Proposition \ref{mutAdd}) forcing with $\D$ and $\Dtree$ each add reals generic for the other, but a $\D$-generic extension is not the same as a $\Dtree$-generic extension.

We do not know if the following trichotomous reinterpretation of their conjecture holds.

\begin{conj}
Let $d$ be a real which is either $\D$-generic or $\Dtree$-generic, and let $x\in V[d]\cap\omega^\omega$ be a new real. Then exactly one of the following holds
\begin{enumerate}
\item $V[x]$ is equivalent to an extension of $V$ by $\D$,
\item $V[x]$ is equivalent to an extension of $V$ by $\Dtree$, or
\item $V[x]$ is a equivalent to an extension of $V$ by $\mathbb{C}$.
\end{enumerate}
\end{conj}


\section{Forcing extensions with no $\preceq$-least dominating real}

Let $V[G]$ be some generic extension of the universe. Given $f,g\in V[G]\cap\omega^\omega$ we write $f\preceq g$ if there are $z_0,z_1\in V\cap\omega^{\nearrow\omega}$ such that $z_0\circ f\circ z_1\leq^* g$. It is easy to see that $\preceq$ gives a preordering on $\omega^\omega$, and furthermore that $f\preceq g$ is equivalent to the existence of $z_0,z_1\in V\cap\omega^{\nearrow\omega}$ such that $f\leq^*z_0\circ g\circ z_1$. Theorem \ref{DomChar} tells us that in the model obtained by adding a nondecreasing Hechler real $d$ we have that $d$ is a $\preceq$-least dominating real: for any dominating real $y$ in $V[d]$ we have $d\preceq y$.

So $V^{\Dnd}$ contains a $\preceq$-least dominating real, and by forcing equivalence so does $V^{\D}$. By suitably modifying the arguments from section 4 one can also show that a $\preceq$-least dominating real is present in $V^\Dtree$. If $d$ is a tree Hechler real (which we may assume is strictly increasing), then a $\preceq$-least real $d^\downarrow$ is defined by the equation $$d^\downarrow(n)=d(k+1)\textrm{ if }n\in[d(k-1),d(k)).$$ The key difference in the argument for $\Dtree$ is that instead of bounding $\phi(s)$ by a function of the length $|s|$, one must be content to bound $\phi$ by a function of $s\upharpoonright|s|-1$.

Need there always exist $\preceq$-least reals in a ccc extension adding a dominating real? The answer is no. Suppose $V$ is a model of $\mathfrak{b}=\mathfrak{d}=\aleph_2$, so that $V$ contains a dominating family $\{z_\alpha:\alpha<\omega_2\}$ well-ordered by $\leq^*$. Let $V[G]$ be an extension satisfying MA and $2^{\aleph_0}=\aleph_3$. There is no $\preceq$-least dominating real $d$ in $V[G]$. If there were then $\{z_\alpha:\alpha<\omega_2\}$, $\{z'_\alpha\circ d\circ z'_\alpha:\alpha<\omega_2\}$ would be an $(\omega_2,\omega_2)$-gap in the sense of \cite{Sch}, but Proposition 90 of that article shows that no such gap exists.

This simple argument does not work if the ground model satisfies CH. Together with Itay Neeman we found a construction to produce an appropriate ccc forcing extension over a model of CH. In fact, this contruction produces a model that not only has no $\preceq$-least dominating reals but also has no $\preceq$-minimal dominating reals; that is, no dominating reals $y_0$ such that whenever $y$ is dominating and $y\preceq y_0$ holds then it follows that $y_0\preceq y$. The disadvantage of the argument is that it uses large cardinals and a fair amount of technical overhead. The main ``trick'' used in the argument is rather nice and may be applicable in other situations, so we will include a proof.

The general idea of the construction is the natural one. We do an $\omega_1$-length finite support iteration of ccc forcings which at each stage places a dominating real that lies below all the dominating reals added so far. The tricky part is in making sure that each iterand in the forcing is actually ccc. To show this we will use an absoluteness argument; this is where the large cardinal assumptions come in.

The forcing we will iterate is a slight modification of the Laver interpolation order given as Definition 13 in \cite{Sch}. Let $\mathcal{F}_0$, $\mathcal{F}_1$ be two subsets of $\omega^\omega$ so that every member of $\mathcal{F}_0$ is dominated by every member of $\mathcal{F}_1$. Say that a real $h\in\omega^\omega$ interpolates $\mathcal{F}_0$ and $\mathcal{F}_1$ if $f_0\leq^*h$ and $h\leq^*f_1$ for every $f_0\in\mathcal{F}_0$ and every $f_1\in\mathcal{F}_1$. The forcing $\mathbb{Q}(\mathcal{F}_0,\mathcal{F}_1)$ consists of conditions $\langle s,f_0,f_1\rangle\in\mathbb{Q}(\mathcal{F}_0,\mathcal{F}_1)$ satisfying
\begin{enumerate}
\item $s\in\omega^{<\omega},f_0\in\mathcal{F}_0,f_1\in\mathcal{F}_1$
\item $(\forall n\geq|s|)f_0(n)\leq f_1(n)$.
\end{enumerate}
and is ordered by $\langle s',f'_0,f'_1\rangle\leq\langle s,f_0,f_1\rangle$ if
\begin{enumerate}
\item $s\subseteq s'$
\item $(\forall n\geq |s|)f_0(n)\leq f'_0(n), f_1'(n)\leq f_1(n)$
\item $(\forall n\in|s'|\setminus |s|)f_0(n)\leq s'(n)\leq f_1(n).$
\end{enumerate}

It is not hard to see that $\mathbb{Q}(\mathcal{F}_0,\mathcal{F}_1)$ adds a real interpolating $\mathcal{F}_0$ and $\mathcal{F}_1$. Unfortunately this forcing may collapse $\omega_1$. The following proposition is very similar to Lemma 45 in \cite{Sch} and can be proved in an identical way.

\begin{prop}\label{extCCC}
Let $Q$ be a transitive model of ZFC containing $\mathcal{F}_0$ and $\mathcal{F}_1$. If there is a transitive model $Q^*$ of ZFC with $Q\subseteq Q^*$ and $\omega_1^Q=\omega_1^{Q^*}$ and so that $Q^*$ contains an interpolant of $\mathcal{F}_0$ and $\mathcal{F}_1$, then in $Q$ the forcing $\mathbb{Q}(\mathcal{F}_0,\mathcal{F}_1)$ is ccc. 
\end{prop}

Given an ordinal $\alpha$ we let $\mathbb{P}_\alpha$ be the $\alpha$-length finite support iterated forcing given by $\langle\dot{\mathbb{Q}}_\beta:\beta<\alpha\rangle$ which we describe as follows. We take $\dot{\mathbb{Q}}_0$ to be Hechler forcing $\mathbb{D}$. Given $0<\beta<\alpha$ with $\mathbb{P}_\beta$ already defined, we take $\dot{\mathbb{Q}}_\beta$ to be a $\mathbb{P}_\beta$-name for $\mathbb{Q}(\mathcal{F}_0,\mathcal{F}_1)$ where $\mathcal{F}_0$ is $V\cap\omega^\omega$ and $\mathcal{F}_1$ is the collection of dominating reals in $V^{\mathbb{P}_\beta}$.

\begin{lemma} \label{techLemma}
Assume there exists a sharp for $\omega$ many Woodin cardinals. Let $\alpha\leq\omega_1$. Then $\mathbb{P}_\alpha$ is ccc.
\end{lemma}

\begin{proof}
By induction on $\alpha$. The base case is trivial and the limit case follows from the inductive assumption and the fact that the iteration has finite support.

Assume by induction that $\mathbb{P}_\beta$ is ccc for all $\beta\leq\alpha$. By identifying nice names for reals with reals we may view $\mathbb{P}_\alpha$ as a subset of $\omega^\omega$. Let $G$ be $\mathbb{P}_\alpha$-generic over $V$. Let $\mathcal{F}_0$ be $V\cap\omega^\omega$ and $\mathcal{F}_1$ be the collection of dominating reals in $V[G]$. We need to show that $\dot{\mathbb{Q}}_\alpha[G]=\mathbb{Q}(\mathcal{F}_0,\mathcal{F}_1)$ is ccc.

Let $M$ be a countable elementary submodel of a large rank initial segment of $V$ with $\alpha\in M$, and let $\pi:M\rightarrow Q$ be the transitive collapse. Note that $\pi(\mathbb{P}_\alpha)=Q\cap\mathbb{P}_\alpha$. Because $\mathbb{P}_\alpha$ is ccc it follows that $G$ is $\pi(\mathbb{P}_\alpha)$-generic over $Q$. Let $\bar{\mathcal{F}}_0=Q[G]\cap\mathcal{F}_0$ (which is just $Q\cap\omega^\omega$) and let $\bar{\mathcal{F}}_1=Q[G]\cap\mathcal{F}_1$ so that $\pi(\dot{\mathbb{Q}}_\alpha)[G]=\mathbb{Q}(\bar{\mathcal{F}}_0,\bar{\mathcal{F}}_1)$.

By elementarity we need only show that $\mathbb{Q}(\bar{\mathcal{F}}_0,\bar{\mathcal{F}}_1)$ is ccc in $Q[G]$. By Proposition \ref{extCCC} it is enough to find a transitive model $Q^*$ extending $Q[G]$ with $\omega_1^{Q[G]}=\omega_1^{Q^*}$ and which contains a real interpolating $\bar{\mathcal{F}_0}$ and $\bar{\mathcal{F}_1}$. Let $\bar{d}\in V\cap\omega^\omega$ be $\mathbb{D}$-generic over $Q$. Then $\bar{d}$ is such an interpolant and so we just need to find an appropriate model $Q^*$ containing it.

Now note that $\mathbb{P}_\alpha$ is a set of reals definable in $L(\mathbb{R})$ (from real parameters). Since $\mathbb{P}_\alpha$ is ccc the collection of maximal antichains of $\mathbb{P}_\alpha$ is also a set of reals definable in $L(\mathbb{R})$. We use $\mathbb{P}_\alpha^{Q[\bar{d}]}$ to refer to $\mathbb{P}_\alpha$ computed relative to $Q[\bar{d}]$. By the inductive hypothesis applied inside $Q[\bar{d}]$ we have that $$Q[\bar{d}]\vDash\textrm{$\mathbb{P}_\alpha$ is ccc.}$$ If we show that $G$ is $\mathbb{P}_\alpha^{Q[\bar{d}]}$-generic over $Q[\bar{d}]$ the proof will be complete, for then $Q[\bar{d}][G]$ can serve as the desired $Q^*$.

\medskip

\noindent\emph{Claim.} Suppose $\varphi$ is a formula and $x\in Q[\overline{d}]\cap\omega^\omega$. Then $$L(\mathbb{R})\vDash\phi(x)\Longleftrightarrow Q[\bar{d}]\vDash ``L(\mathbb{R})\vDash\varphi(x)".$$

\begin{proof}
Here is where the large cardinal machinery comes in. We just give a sketch:

The theory of $L({\mathbb R})$ with parameter $x$ reduces to the theory of 
any model $(N;{\mathcal E})$, where ${\mathcal E}$ is a class of extenders rich enough to witness 
that $N$ has a sharp for $\omega$ Woodin cardinals, $x\in N$, and $(N;{\mathcal E})$ is 
iterable for iteration trees using extenders from ${\mathcal E}$. (See section 4 of \cite{Steel} 
where full iterability is used, or section 3 of \cite{Neeman} which uses $\omega$-iterability.) 
It is therefore enough to find, in $Q[\bar{d}]$, a model of this kind which is 
$\omega$-iterable in both $Q[\bar{d}]$ and $V$.

Let $V^Q_\theta$ be a rank initial segment of $Q$ large enough to contain a sharp for 
$\omega$ Woodin cardinals. Working in $Q$ using the genericity iterations of \cite{Neeman} one 
can find a countable model $N$, which embeds into $V^Q_\theta$, and a generic extension 
$N[g]$ of $N$ collapsing its first Woodin cardinal, so that $x\in N[g]$. Since $N$ embeds 
into $V^Q_\theta$, and since $Q[\bar{d}]$ is a small generic extension of $Q$, $N$ is 
$\omega$-iterable in $Q[\bar{d}]$. Since $V^Q_\theta$ embeds into a rank initial segment 
of $V$, so does $N$, and hence $N$ is $\omega$-iterable in $V$. The iterability transfers 
to $(N[g],{\mathcal E})$, where ${\mathcal E}$ consists of the natural extensions of extenders in 
$N$. $(N[g],{\mathcal E})$ is then $\omega$-iterable in both $Q[\bar{d}]$ and $V$, contains 
$x$, and has a sharp for $\omega$ Woodin cardinals, as required.
\end{proof}

Let $\mathcal{A}\in Q[\bar{d}]$ be such that $$Q[\bar{d}]\vDash``\mathcal{A}\textrm{ is a maximal antichain of $\mathbb{P}_\alpha$}".$$ As we observed above $\mathbb{P}_\alpha^{Q[\bar{d}]}$ is ccc from the point of view of $Q[\bar{d}]$ and so we may view $\mathcal{A}$ as an element of $\omega^\omega$. Applying the claim we have (in $V$) that $\mathcal{A}$ is a maximal antichain of $\mathbb{P}_\alpha$. Thus $G\cap\mathcal{A}\not=\varnothing$ and we are done.

\end{proof}

\begin{thm}
Suppose $V$ is a model of CH which contains a sharp for $\omega$ many Woodin cardinals. Then $\mathbb{P}_{\omega_1}$ is a ccc forcing which adds a dominating real but no $\preceq$-minimal dominating real.
\end{thm}

\begin{proof}
From Lemma \ref{techLemma} we get that $\mathbb{P}_{\omega_1}$ is ccc. Because CH holds in the ground model any dominating real $d$ added by $\mathbb{P}_{\omega_1}$ is added at some countable stage $\mathbb{P}_\alpha$. Every real of the form $z_0\circ d\circ z_1$ belongs to $V^{\mathbb{P}_\alpha}$ and so the forcing $\dot{\mathbb{Q}}_\alpha$ adds a dominating real $h$ below them all. Thus $d$ is not a minimal dominating real for the preordering $\preceq$ in $V^{\mathbb{P}_\alpha}$.
\end{proof}

\bibliographystyle{alpha}
\bibliography{UnboundedDominating}

\end{document}